\documentclass{article}
\usepackage{amsfonts}
\usepackage{amsmath}

\newtheorem{theorem}{Theorem}

\newtheorem{case}[theorem]{Case}

\newtheorem{corollary}[theorem]{Corollary}

\newtheorem{definition}[theorem]{Definition}
\newtheorem{example}[theorem]{Example}

\newtheorem{remark}[theorem]{Remark}

\newenvironment{proof}[1][Proof]{\noindent\textbf{#1.} }{\ \rule{0.5em}{0.5em}}
\input{tcilatex}

\begin{document}

\begin{center}
\textbf{GENERALIZATION OF SPECIAL FUNCTIONS AND ITS APPLICATIONS TO }

\textbf{MULTIPLICATIVE AND ORDINARY FRACTIONAL DERIVATIVES}

\bigskip

\textbf{Ali \"{O}zyap{\i }c{\i }\footnote{%
\textquotedblleft The author(s) declare(s) that there is no conflict of
interest regarding the publication of this paper.\textquotedblright\ }}

Faculty of Engineering, Cyprus International University, Nicosia, Mersin
10,Turkey (aozyapici@ciu.edu.tr)

\textbf{Yusuf Gurefe}

Department of Econometrics , Usak University, Usak, Turkey
(yusuf.gurefe@usak.edu.tr)

\textbf{Emine M\i s\i rl\i }

Department of Mathematics, Ege University, Izmir, Turkey
(emine.misirli@ege.edu.tr)
\end{center}



\textit{Corresponding author:} ali.ozyapici@emu.edu.tr

\begin{center}
\textbf{Abstract}
\end{center}

The goal of this paper is to extend the classical and multiplicative
fractional derivatives. For this purpose, it is introduced the new extended
modified Bessel function and also given an important relation between this
new function $I_{\upsilon }(q;x)$ and the confluent hypergeometric function $%
{_{1}F_{1}}(\alpha ,\beta ,x)$. Besides, it is used to generalize the
hypergeometric, the confluent hypergeometric and the extended beta functions
by using the new extended modified Bessel function. Also, the asymptotic
formulae and the generating function of the extended modified Bessel
function are obtained. The extensions of classical and multiplicative
fractional derivatives are defined via extended modified Bessel function
and, first time the fractional derivative of rational functions is
explicitly given via complex partial fraction decomposition.

\section{Introduction}

\subsection{\protect\bigskip Generalized Special Functions}

Especially, in the last two decades, several generalizations of the
well-known special functions have been studied by different authors. In
1997, Chaudhry \cite{CQRZ} have introduced the extension of Euler's beta
function by 
\begin{align}
B_{p}(x,y)& =\int_{0}^{1}t^{x-1}(1-t)^{y-1}e^{-\frac{p}{t(1-t)}}dt.
\label{B1} \\
(Re(p)& >0,Re(x)>0,Re(y)>0)  \notag
\end{align}%
It is clear that the special case $p=0$ gives the Euler's beta function $%
B_{0}(x,y)=B(x,y)$.

Then, the authors in (\cite{OMA}) extended beta functions and hypergeometric
functions as%
\begin{align}
B_{p}^{\left( \alpha ,\beta \right) }\left( x,y\right) &
=\int_{0}^{1}t^{x-1}(1-t)^{y-1}{}_{1}F_{1}(\alpha ;\beta ;\frac{-p}{t(1-t)}%
)dt  \label{B2} \\
(Re(p)& >0,Re(x)>0,Re(y)>0,Re(\alpha )>0,Re(\beta )>0).  \notag
\end{align}

Lee et al. in (\cite{LRPK}) introduced the more generalized Beta type
function as follows:

\begin{align}
B_{p}^{(\alpha ,\beta ;m)}(x,y)&
=\int_{0}^{1}t^{x-1}(1-t)^{y-1}{}_{1}F_{1}(\alpha ;\beta ;\frac{-p}{%
t^{m}(1-t)^{m}})dt.  \label{B5} \\
(Re(p)& >0,Re(x)>0,Re(y)>0,Re(\alpha )>0,Re(\beta )>0)  \notag
\end{align}%
Consequently, Luo et. al. in (\cite{LMA}) \ generalized extended beta
function (\ref{B5}) (as well as (\ref{B1}) and (\ref{B2})) by introducing 
\begin{align}
B_{b;\rho ;\lambda }^{\left( \alpha ,\beta \right) }(x,y)&
=\int_{0}^{1}t^{x-1}(1-t)^{y-1}{}_{1}F_{1}(\alpha ;\beta ;\frac{-b}{t^{\rho
}\left( 1-t\right) ^{\lambda }})dt.  \label{B4} \\
(\rho & \geq 0,\lambda \geq 0,\min \{\func{Re}(\alpha ),\func{Re}(\beta
)\}>0,\func{Re}(x)>-\func{Re}(\rho \alpha ),\func{Re}(y)>-\func{Re}(\lambda
\alpha ))  \notag
\end{align}

\bigskip Recently, \ Parmar in (\cite{P}) introduced very interesting
special function consisting Bessel function of second kind as 
\begin{align}
B_{v}(x,y;p)& =\sqrt{\frac{2p}{\pi }}\int_{0}^{1}t^{x-\frac{3}{2}}(1-t)^{y-%
\frac{3}{2}}{}K_{v+\frac{1}{2}}(\frac{p}{t\left( 1-t\right) })dt,  \label{P1}
\\
(\func{Re}(p)& >0).  \notag
\end{align}%
Finally, we refer the papers (\cite{NVG}) and (\cite{O}) for more properties
of extended Gauss hypergeometric and extended confluent hypergeometric
functions.

In this paper, we introduce extended special functions as generalizations of
modified Bessel Functions, Beta functions, hypergeometric functions and
confluent hypergeometric functions. We would like to mention an interesting
remark from Qadir \cite{Q} who explains the importance of generalization of
the special functions as "Notice that the generalization of the other
special functions has proved even more useful than the separate special
functions themselves". We refer the paper \cite{Q} for more details about
generalization of the special functions.

Consequently, we define extended fractional derivative and extended
multiplicative fractional derivative.

\subsection{Multiplicative (geometric) Calculus and Fractional Derivatives}

\bigskip Multiplicative calculus has improved rapidly over the past 10
years. In this period, superiority of the multiplicative calculus over
ordinary calculus was proved by many studies. The most significant among
these studies are \cite{FA} in Biomedical Image Analysis, \cite{BR} in
complex analysis, \cite{FP} in growth phenomena, \cite{MG}, \cite{ROM}, \cite%
{A} and \cite{AAB} in numerical analysis, \cite{BMYO} in actuarial science,
finance, demography etc., \cite{ECL} in biology, recently \cite{HAI} in
accounting. In order for multiplicative calculus to be used efficiently in
all respects more studies needs to be done in various field. Recently,
multiplicative Laplace transform%
\begin{align}
L_{m}\left\{ f\left( t\right) \right\} & =\exp \left( \int_{0}^{\infty
}e^{-st}\ln \left( f\left( t\right) \right) dt\right) \\
(f\left( t\right) & \in \left[ 0,\infty \right) ).  \notag
\end{align}

has been introduced and \ applied in optics in \cite{YCG}.

\bigskip First application of the multiplicative calculus to fractional
derivative is executed by Abdeljawad and Grossman in \cite{AG}. In this
paper, the Riemann-Lioville fractional integral of order $\alpha \in 
\mathbb{C}
$ has been defined as%
\begin{equation}
\left( aI_{\ast }^{\alpha }f\right) \left( x\right) =e^{\frac{1}{\Gamma
\left( \alpha \right) }\int_{a}^{x}\left( x-t\right) ^{\alpha -1}(\ln \circ
f)(x)\,dx},x>a.  \label{R-L}
\end{equation}

\bigskip In this article multiplicative fractional derivative (\ref{R-L}) is
extended. The connection between multiplicative and ordinary fractional
derivatives is presented with assertion of some essential characteristics of
fractional derivative. As a result, the generalized ordinary fractional
derivative introduced in section 3 is applied to introduce multiplicative
generalized fractional derivative.

\bigskip

\section{Extension of Special Functions}

In this section, we introduce special functions which will be generalization
of the functions (\ref{B1})-(\ref{P1}).

\subsection{Extended Modified Bessel Function}

We here introduce new extended of modified Bessel functions as follows.

\begin{definition}
The function 
\begin{align}
I_{v}(q;x)& =\frac{\left( \frac{x}{2}\right) ^{v}}{\sqrt{\pi }\Gamma \left(
v+\frac{1}{2}\right) }\int_{-1}^{1}(1-t^{2})^{v-\frac{1}{2}}{}(1-t)^{q-\frac{%
1}{2}}\exp \left( -x\left( t-1\right) \right) dt,  \label{EB} \\
(\func{Re}(v+q)& >0,\func{Re}(v)>\frac{-1}{2})  \notag
\end{align}%
is called extended modified Bessel function whenever integral exists.
\end{definition}

It is clear that the function (\ref{EB}) reduce to Bessel function when $q=%
\frac{1}{2}.$ Explicitly, $I_{v}(\frac{1}{2};x)=\exp (x)I_{v}(x).$

\begin{corollary}
We have the following integral representation for $I_{\upsilon }(q;x)$: 
\begin{eqnarray}
I_{\upsilon }(q;x) &=&\frac{\left( \frac{x}{2}\right) ^{\upsilon
}2^{2\upsilon +q-\frac{1}{2}}}{\sqrt{\pi }\ \Gamma \left( \upsilon +\frac{1}{%
2}\right) }\int_{0}^{1}t^{\upsilon +q-1}\left( 1-t\right) ^{\upsilon -\frac{1%
}{2}}\exp (2xt)\ \emph{d}t.  \label{eq.7} \\
(\func{Re}(v+q) &>&0,\func{Re}(v)>\frac{-1}{2})
\end{eqnarray}

\begin{proof}
By using the transformation $t\rightarrow 1-2t,$ the statement can be obtain.
\end{proof}
\end{corollary}

\begin{theorem}
The extended modified Bessel function $I_{\upsilon }(q;x)$ has power series
representation as follows: 
\begin{equation}
I_{\upsilon }(q;x)=\frac{\left( \frac{x}{2}\right) ^{\upsilon }}{2^{q+\frac{1%
}{2}}}\sum_{n=0}^{\infty }\frac{\Gamma (2\upsilon +2q+2n)}{\Gamma (\upsilon
+q+n+\frac{1}{2})\ \Gamma (2\upsilon +q+n+\frac{1}{2})\ n!}\left( \frac{x}{2}%
\right) ^{n}.  \label{eq.8}
\end{equation}

\begin{proof}
From the representation (\ref{eq.7}), we can write the following relation
consisting of the power series of the function $\exp (2xt)$ 
\begin{equation}
\begin{split}
\int_{0}^{1}t^{\upsilon +q-1}\left( 1-t\right) ^{\upsilon -\frac{1}{2}}\exp
(2xt)\ \emph{d}t& =\sum_{n=0}^{\infty }\frac{(2x)^{n}}{n!}%
\int_{0}^{1}t^{\upsilon +q+n-1}\left( 1-t\right) ^{\upsilon -\frac{1}{2}}\ 
\emph{d}t \\
& =\sum_{n=0}^{\infty }B\left( \upsilon +q+n,\upsilon +\frac{1}{2}\right) 
\frac{(2x)^{n}}{n!} \\
& =\sum_{n=0}^{\infty }\frac{\Gamma (\upsilon +q+n)\ \Gamma \left( \upsilon +%
\frac{1}{2}\right) }{\Gamma \left( 2\upsilon +q+n+\frac{1}{2}\right) }\frac{%
(2x)^{n}}{n!}.
\end{split}
\label{eq.9}
\end{equation}%
Using the Legendre's duplication formula, we get 
\begin{equation}
\begin{split}
\Gamma (\upsilon +q+n)& =\Gamma \left( \upsilon +q+n-\frac{1}{2}+\frac{1}{2}%
\right) \\
& =\frac{\sqrt{\pi }}{2^{2\upsilon +2q+2n-2}}\frac{\Gamma (2\upsilon
+2q+2n-1)}{\Gamma \left( \upsilon +q+n-\frac{1}{2}\right) } \\
& =\frac{\sqrt{\pi }}{2^{2\upsilon +2q+2n-1}}\frac{\Gamma (2\upsilon +2q+2n)%
}{\Gamma \left( \upsilon +q+n+\frac{1}{2}\right) }.
\end{split}
\label{eq.10}
\end{equation}%
Substituting equations (\ref{eq.9}) and (\ref{eq.10}) into equation (\ref%
{eq.7}), we obtain 
\begin{equation}
I_{\upsilon }(q;x)=\frac{\left( \frac{x}{2}\right) ^{\upsilon }}{2^{q+\frac{1%
}{2}}}\sum_{n=0}^{\infty }\frac{\Gamma (2\upsilon +2q+2n)}{\Gamma (\upsilon
+q+n+\frac{1}{2})\ \Gamma (2\upsilon +q+n+\frac{1}{2})\ n!}\left( \frac{x}{2}%
\right) ^{n}.  \label{eq.11}
\end{equation}
\end{proof}
\end{theorem}

\begin{theorem}
The relation between the extended modified Bessel function $I_{\upsilon
}(q;x)$ and the confluent hypergeometric function $_{1}F_{1}(\alpha ,\beta
,x)$ is 
\begin{equation}
I_{\upsilon }(q;x)=\frac{\left( \frac{x}{2}\right) ^{\upsilon }\
2^{2\upsilon +q-\frac{1}{2}}\ \Gamma (\upsilon +q)}{\sqrt{\pi }\ \Gamma
(2\upsilon +q+\frac{1}{2})}{_{1}F_{1}}\left( \upsilon +q,2\upsilon +q+\frac{1%
}{2},2x\right) .  \label{I-F}
\end{equation}
\end{theorem}

\textbf{Proof.} Recall that 
\begin{equation}
I_{\upsilon }(q;x)=\frac{\left( \frac{x}{2}\right) ^{\upsilon }2^{2\upsilon
+q-\frac{1}{2}}}{\sqrt{\pi }\ \Gamma \left( \upsilon +\frac{1}{2}\right) }%
\int_{0}^{1}t^{\upsilon +q-1}\left( 1-t\right) ^{\upsilon -\frac{1}{2}}\exp
(2xt)\ dt.  \label{eq.13}
\end{equation}%
Consider the representation of the function $_{1}F_{1}(\alpha ,\beta ,x)$ as 
\begin{equation}
{_{1}F_{1}}(\alpha ,\beta ,2x)=\frac{\Gamma (\beta )}{\Gamma (\alpha )\
\Gamma (\beta -\alpha )}\int_{0}^{1}t^{\alpha -1}\left( 1-t\right) ^{\beta
-\alpha -1}\exp (2xt)\ dt.  \label{eq.14}
\end{equation}%
Hence, the special cases $\alpha =\upsilon +q$ and $\beta =2\upsilon +q+%
\frac{1}{2}$ give us a new relation between the extended modified Bessel
function and the confluent hypergeometric function as follows: 
\begin{equation}
I_{\upsilon }(q;x)=\frac{\left( \frac{x}{2}\right) ^{\upsilon }\
2^{2\upsilon +q-\frac{3}{2}}\ \Gamma (\upsilon +q)}{\sqrt{\pi }\ \Gamma
\left( 2\upsilon +q+\frac{1}{2}\right) }{_{1}F_{1}}\left( \upsilon
+q,2\upsilon +q+\frac{1}{2},2x\right)  \label{eq.15}
\end{equation}%
which proves the theorem.

The relation (\ref{I-F}) provides a wide range of applications of the
function (\ref{eq.7}). Since $I_{\upsilon }(q;x)$ represents both modified
Bessel and confluent hypergeometric functions, the special function $%
I_{\upsilon }(q;x)$ can effectively used to generalize many special
functions.

\begin{example}
For some values of $\upsilon $ and $q$ we can easily write the following
relations 
\begin{equation}
I_{1/2}(1/2;x)=e^{x}I_{1/2}(x)=\sqrt{\frac{2}{\pi x}}e^{x}\sinh x,  \notag
\end{equation}%
\begin{equation}
I_{0}(1/2;x)=e^{x}J_{0}(x),\ \ \ I_{1}(1/2;x)=e^{x}J_{1}(x),\ \ \
I_{0}(3/2;x)=e^{x}(J_{0}(x)+J_{1}(x)),  \notag
\end{equation}%
\begin{equation}
I_{\upsilon }(1/2;x)=e^{x}J_{\upsilon }(x),\ \ \ Re(\upsilon )>-\frac{1}{2},
\notag
\end{equation}%
\begin{equation}
I_{0}(q;x)=\frac{2^{q-\frac{1}{2}}}{\sqrt{\pi }}\ {_{1}F_{1}}(q,q+\frac{1}{2}%
,2x),\ \ \ Re(q)>0.  \notag
\end{equation}%
\begin{equation}
I_{1/2}(q;x)=(-x)^{-q-\frac{1}{2}}\sqrt{\frac{x}{2\pi }}\ \left( \Gamma (q+%
\frac{1}{2})-\Gamma (q+\frac{1}{2},-2x)\right) ,\ \ \ Re(q)>-\frac{1}{2}. 
\notag
\end{equation}
\end{example}

The next theorem deals with an asymptotic formula of extended modified
Bessel function.

\begin{theorem}
The special function $I_{\upsilon }(q;x)$ as $x\rightarrow \infty $
approaches to%
\begin{equation*}
I_{\upsilon }(q;x)\rightarrow \frac{\Gamma \left( v+q\right) }{\sqrt{2\pi }%
x^{q}\Gamma \left( v+\frac{1}{2}\right) }.
\end{equation*}

\begin{proof}
Consider the integral representation of the $I_{\upsilon }(q;x)$ as%
\begin{equation*}
I_{\upsilon }(q;x)=\frac{\left( \frac{x}{2}\right) ^{\upsilon }2^{2\upsilon
+q-\frac{1}{2}}}{\sqrt{\pi }\ \Gamma \left( \upsilon +\frac{1}{2}\right) }%
\int_{0}^{1}t^{\upsilon +q-1}\left( 1-t\right) ^{\upsilon -\frac{1}{2}}\exp
(2xt)\ dt.
\end{equation*}%
Let $I=\int_{0}^{1}t^{\upsilon +q-1}\left( 1-t\right) ^{\upsilon -\frac{1}{2}%
}\exp (2xt)\ dt.$ By using the substitution $t=\frac{u}{u-x},$ the integral $%
I$ will be 
\begin{eqnarray*}
I &=&\int_{0}^{\infty }\frac{u^{\upsilon +q-1}}{\left( u-x\right) ^{\upsilon
+q-1}}\left( 1-\frac{u}{u-x}\right) ^{\upsilon -\frac{1}{2}}\exp (\frac{2xu}{%
u-x})\ \left( \frac{-x}{\left( u-x\right) ^{2}}\right) du \\
&=&\int_{0}^{\infty }\frac{u^{\upsilon +q-1}x^{v+\frac{1}{2}}}{x^{2v+q+\frac{%
1}{2}}\left( 1-\frac{u}{x}\right) ^{2\upsilon +q+\frac{1}{2}}}\exp (\frac{2xu%
}{u-x})\ du \\
&=&\frac{1}{x^{v+q}}\int_{0}^{\infty }u^{\upsilon +q-1}\exp (-2u)du
\end{eqnarray*}%
where $\frac{u}{x}\rightarrow 0$ and $\frac{2xu}{u-x}\rightarrow -2u$ for
large number $x.$ Since $\int_{0}^{\infty }u^{\upsilon +q-1}\exp
(-2u)du=2^{-v-q}\Gamma \left( v+q\right) ,$ we have%
\begin{equation*}
I_{\upsilon }(q;x)=\frac{\left( \frac{x}{2}\right) ^{\upsilon }2^{2\upsilon
+q-\frac{1}{2}}}{\sqrt{\pi }\ \Gamma \left( \upsilon +\frac{1}{2}\right) }%
\frac{2^{-v-q}\Gamma \left( v+q\right) }{x^{v+q}}
\end{equation*}%
which proves the theorem.
\end{proof}
\end{theorem}

\begin{remark}
If we consider the relation%
\begin{equation*}
I_{v}(\frac{1}{2};x)=e^{x}I_{v}(x),
\end{equation*}%
the corresponding asymptotic formula of modified Bessel function of first
kind can easily be derived as%
\begin{equation*}
I_{v}(x)\rightarrow \frac{e^{x}}{\sqrt{2\pi x}},\text{ }x\rightarrow \infty .
\end{equation*}
\end{remark}

\begin{theorem}
For $\left\vert \frac{2}{z}\right\vert <1$, the following generating
function holds true:%
\begin{equation}
\sum_{n=-\infty }^{\infty }I_{n+\frac{1}{2}}(-n+\frac{1}{2};x)z^{n}=\sqrt{%
\frac{2}{\pi x}}\left( \frac{ze^{xz}}{z-2}\right) .
\end{equation}
\end{theorem}

\begin{proof}
By using Legendre's duplication formula, the series representation of $%
I_{v}(q;x)$ can be given as%
\begin{equation*}
I_{\upsilon }(q;x)=\sum_{n=0}^{\infty }\frac{2^{v+q+n-\frac{1}{2}}\Gamma
(\upsilon +q+n)}{\ \sqrt{\pi }\Gamma (2\upsilon +q+n+\frac{1}{2})\ n!}%
x^{n+v}.
\end{equation*}%
Consequently, 
\begin{eqnarray}
\sum_{n=-\infty }^{\infty }I_{n+\frac{1}{2}}(-n+\frac{1}{2};x)z^{n}
&=&\sum_{n=-\infty }^{\infty }\left( \sum_{k=0}^{\infty }\frac{\Gamma \left(
k+1\right) 2^{k+\frac{1}{2}}}{\sqrt{\pi }k!\Gamma \left( n+k+2\right) }%
x^{n+k+\frac{1}{2}}\right) z^{n}  \notag \\
&=&\sum_{n=-\infty }^{\infty }\left( \sum_{m=n+k+1}^{\infty }\frac{2^{k+%
\frac{1}{2}}}{\sqrt{\pi }\Gamma \left( m+1\right) }x^{m-\frac{1}{2}}\right)
z^{n}  \notag \\
&=&\sqrt{\frac{2}{\pi x}}\left( \sum_{m=0}^{\infty }\frac{\left( xz\right)
^{m}}{m!}.\sum_{k=0}^{\infty }2^{k}z^{-k}\right) .
\end{eqnarray}%
Since $\left\vert \frac{2}{z}\right\vert <1,$ the geometric series 
\begin{equation*}
\sum_{k=0}^{\infty }2^{k}z^{-k}=\frac{1}{1-\frac{2}{z}}.
\end{equation*}%
Hence,%
\begin{equation*}
\sum_{n=-\infty }^{\infty }I_{n+\frac{1}{2}}(-n+\frac{1}{2};x)z^{n}=\sqrt{%
\frac{2}{\pi x}}e^{xz}\frac{z}{z-2}.
\end{equation*}
\end{proof}

Next, we attempt to find generating functions involving the special function 
$I_{v}(q;x),$mainly motivated by the paper of Agarwal et al. (\cite{ACP}).

\begin{theorem}
For $v,q\in 
\mathbb{C}
$, the following generating function holds true:%
\begin{equation}
\sum_{k=0}^{\infty }\frac{1}{\Gamma \left( k+1\right) }I_{v-k}(q+2k;x)t^{k}=%
\left( 1-\frac{2t}{x}\right) ^{-v-q}I_{v}(q;\frac{x^{2}}{x-2t}).  \label{GF}
\end{equation}
\end{theorem}

\begin{proof}
By using Legendre's duplication formula, the series representation of $%
I_{v}(q;x)$ can be given as%
\begin{equation*}
I_{\upsilon }(q;x)=\sum_{n=0}^{\infty }\frac{2^{v+q+n-\frac{1}{2}}\Gamma
(\upsilon +q+n)}{\ \sqrt{\pi }\Gamma (2\upsilon +q+n+\frac{1}{2})\ n!}%
x^{n+v}.
\end{equation*}%
Consequently, by a little simplifications,%
\begin{equation}
\sum_{k=0}^{\infty }\frac{1}{\Gamma \left( k+1\right) }I_{v-k}(q+2k;x)t^{k}=%
\sum_{n=0}^{\infty }\frac{2^{v+q+n-\frac{1}{2}}x^{n+v}}{\ \sqrt{\pi }\Gamma
(2\upsilon +q+n+\frac{1}{2})\ n!}\cdot \sum_{k=0}^{\infty }\frac{\Gamma
(\upsilon +q+n+k)2^{k}t^{k}}{\ \ k!x^{k}}.  \label{GF-1}
\end{equation}%
Since \ $\sum_{k=0}^{\infty }\frac{\Gamma (\lambda +k)t^{k}}{\ \ k!\Gamma
\left( \lambda \right) }=\left( 1-t\right) ^{\lambda }$ for $\lambda \in 
\mathbb{C}
,$%
\begin{equation*}
\sum_{k=0}^{\infty }\frac{\Gamma (\upsilon +q+n+k)2^{k}t^{k}}{\ \ k!x^{k}}%
=\Gamma (\upsilon +q+n)\sum_{k=0}^{\infty }\frac{\Gamma (\upsilon
+q+n+k)\left( \frac{2t}{x}\right) ^{k}}{\Gamma (\upsilon +q+n)\ \ k!}.
\end{equation*}%
Therefore the infinite sum (\ref{GF-1}) becomes%
\begin{eqnarray*}
\sum_{k=0}^{\infty }\frac{1}{\Gamma \left( k+1\right) }I_{v-k}(q+2k;x)t^{k}
&=&\sum_{n=0}^{\infty }\frac{2^{v+q+n-\frac{1}{2}}x^{n+v}\Gamma (\upsilon
+q+n)}{\ \sqrt{\pi }\Gamma (2\upsilon +q+n+\frac{1}{2})\ n!}\cdot \left[
\left( 1-\frac{2t}{x}\right) ^{-\left( v+q+n\right) }\right] \\
&=&\sum_{n=0}^{\infty }\frac{2^{v+q+n-\frac{1}{2}}x^{v}\left( \frac{x}{1-%
\frac{2t}{x}}\right) ^{n}\Gamma (\upsilon +q+n)}{\ \sqrt{\pi }\Gamma
(2\upsilon +q+n+\frac{1}{2})\ n!}\cdot \left[ \left( 1-\frac{2t}{x}\right)
^{-\left( v+q\right) }\right]
\end{eqnarray*}%
which gives the generating function given in (\ref{GF}).
\end{proof}

We aim to continue to generate a new generating function involving confluent
hypergeometric function ${_{1}F_{1}}(\alpha ,\beta ,x)$ via generating
function(\ref{GF}).

\begin{theorem}
For $\alpha ,\beta \in 
\mathbb{C}
$, the following generating function holds true:%
\begin{equation}
\sum_{k=0}^{\infty }\frac{\Gamma \left( k+\alpha \right) }{\Gamma \left(
k+1\right) }{_{1}F_{1}}\left( \alpha +k,\beta ,x\right) z^{k}=\left(
1-z\right) ^{-\alpha }\Gamma \left( \alpha \right) {_{1}F_{1}}\left( \alpha
,\beta ,\frac{x}{1-z}\right) .  \label{GFC}
\end{equation}

\begin{proof}
Use the generating function (\ref{GF}) together with relation (\ref{I-F})
and set $\alpha \rightarrow \upsilon +q,\beta \rightarrow 2\upsilon +q+\frac{%
1}{2},z\rightarrow \frac{2t}{x}.$
\end{proof}
\end{theorem}

For one interesting reference paper about generating functions of special
functions, we refer the paper Cohl. et al. (\cite{CMV}).

\subsection{\protect\bigskip Extended Hypergeometric, Confluent
Hypergeometric and Beta Functions via extended modified Bessel Function}

We use the new extension (\ref{EB}) of extended beta function to generalize
the hypergeometric, confluent hypergeometric and extended beta functions as
follows.

\begin{definition}
The extended beta-hypergeometric function \bigskip 
\begin{align}
B_{v,q}^{\left( \mu ,\sigma \right) }(x,y;p)& =\sqrt{\frac{2}{\pi }}%
\int_{0}^{1}t^{x}(1-t)^{y}I_{v+\frac{1}{2}}(q;\frac{-p}{t^{\mu }\left(
1-t\right) ^{\sigma }})dt,  \label{E1} \\
(\func{Re}(p)& >0,\text{ }\mu ,\sigma \geq 0,\min \{\func{Re}\left( v+q+%
\frac{1}{2}\right) ,\func{Re}\left( 2v+q+\frac{3}{2}\right) \}>0,  \notag \\
\func{Re}(x+\mu q)& >-1,\func{Re}(y+\sigma q)>-1)
\end{align}%
is defined.
\end{definition}

\begin{remark}
The necessary conditions for existence of integral given in (\ref{E1}) can
also be derived by using the relation (\ref{I-F}) and the paper (\cite{LMA},
pp. 633, theorem 2.1).
\end{remark}

It is clear that the new extension (\ref{E1}) reduces to many defined
special functions as

\begin{case}
function (\ref{B1}) in the paper\ when $v=0,q=\frac{1}{2}$ and $\mu ,\sigma
=1,$
\end{case}

\begin{case}
generalized of extended beta function in Lee et al. (\cite{LRPK},\textit{\
pp. 189, equations (1.13)}) when $v=0,q=\frac{1}{2}$ and $\mu ,\sigma =m,$
\end{case}

\begin{case}
function (\ref{B4}) when $q=2\alpha -\beta +\frac{1}{2},v=\beta -\alpha -%
\frac{1}{2}$ and $\mu =\rho ;\sigma =\lambda .$
\end{case}

\begin{case}
function (\ref{P1}) in the paper by using the relation $K_{v}(x)=\frac{\pi
\exp \left( -x\right) }{2\sin v\pi }\left( I_{-v}(\frac{1}{2};x)-I_{v}(\frac{%
1}{2};x)\right) $ where $v\notin 
\mathbb{Z}
$.
\end{case}

Consequently, we use the function (\ref{E1}) to extend the hypergeometric
functions and beta functions as follows:

\begin{eqnarray}
F_{v,q;p}^{(\mu ,\sigma )}(a,b;c;z) &=&\sum_{n=0}^{\infty }(a)_{n}\frac{%
B_{v,q}^{(\mu ,\sigma )}(b+n,c-b;p)}{B(b,c-b)}\frac{z^{n}}{n!},  \label{F1}
\\
(\func{Re}(p) &>&0,\text{ }\left\vert z\right\vert <1,\text{ }\min \{\func{Re%
}\left( v+q+\frac{1}{2}\right) ,\func{Re}\left( 2v+q+\frac{3}{2}\right) \},%
\text{ }  \notag \\
\func{Re}(\mu ),\func{Re}(\sigma ) &\geq &0,\func{Re}(c)>\func{Re}(b)>0,%
\func{Re}(a)>0)  \notag
\end{eqnarray}%
and

\begin{eqnarray}
\Phi _{v,q;p}^{(\mu ,\sigma )}(b;c;z) &=&\sum_{n=0}^{\infty }\frac{%
B_{v,q}^{(\mu ,\sigma ))}(b+n,c-b;p)}{B(b,c-b)}\frac{z^{n}}{n!},  \label{F2}
\\
(\func{Re}(p) &>&0,\text{ }\min \{\func{Re}\left( v+q+\frac{1}{2}\right) ,%
\func{Re}\left( 2v+q+\frac{3}{2}\right) \},  \notag \\
\func{Re}(\mu ),\func{Re}(\sigma ) &\geq &0,\func{Re}(c)>\func{Re}(b)>0,%
\func{Re}(a)>0).  \notag
\end{eqnarray}

\begin{theorem}
\bigskip The special functions (\ref{F1}) and (\ref{F2}) , respectively, has
the following integral representation%
\begin{eqnarray*}
{\small F}_{v,q;p}^{(\mu ,\sigma )}{\small (a,b;c;z)} &{\small =}&\text{$%
\frac{\sqrt{\frac{2}{\pi }}}{B(b,c-b)}$}\int_{0}^{1}{\small t}^{b}{\small %
(1-t)}^{c-b}{\small (1-zt)}^{-a}{\small I}_{v+\frac{1}{2}}{\small (q;}\frac{%
-p}{t^{\mu }\left( 1-t\right) ^{\sigma }}{\small )dt,} \\
(\func{Re}(p) &>&0,\text{ }\left\vert \arg (1-z)\right\vert <\pi ,\text{ }%
\func{Re}(\mu ),\func{Re}(\sigma )\geq 0,\func{Re}(c)>\func{Re}(b)>0, \\
\func{Re}(v+q) &>&0,\func{Re}(v)>\frac{-1}{2},\func{Re}(a)>0).
\end{eqnarray*}%
and
\end{theorem}

\begin{eqnarray}
\Phi _{v,q;p}^{(\mu ,\sigma )}(b;c;z) &=&\frac{\sqrt{\frac{2}{\pi }}}{%
B(b,c-b)}\int_{0}^{1}t^{b}(1-t)^{c-b}e^{zt}I_{v+\frac{1}{2}}(q;\frac{-p}{%
t^{\mu }\left( 1-t\right) ^{\sigma }})dt,  \label{F2I} \\
(\func{Re}(p) &>&0,\text{ }\func{Re}(\mu ),\func{Re}(\sigma )\geq 0,\func{Re}%
(c)>\func{Re}(b)>0,  \notag \\
\func{Re}(v+q) &>&0,\func{Re}(v)>\frac{-1}{2},\func{Re}(a)>0).
\end{eqnarray}

\begin{proof}
\bigskip Substituting the function (\ref{E1}) with $x\rightarrow
b+n,y\rightarrow c-b$ into function (\ref{F1}), we have after interchanging
the order of summation and integration which is guaranteed%
\begin{eqnarray*}
F_{v,q;p}^{(\mu ,\sigma )}(a,b;c;z) &=&\frac{\sqrt{\frac{2}{\pi }}}{B(b,c-b)}%
\int_{0}^{1}t^{b}(1-t)^{c-b}I_{v+\frac{1}{2}}(q;\frac{-p}{t^{\mu }\left(
1-t\right) ^{\sigma }})\sum_{n=0}^{\infty }(a)_{n}\frac{\left( zt\right) ^{n}%
}{n!}dt \\
&=&\frac{\sqrt{\frac{2}{\pi }}}{B(b,c-b)}%
\int_{0}^{1}t^{b}(1-t)^{c-b}(1-zt)^{-a}I_{v+\frac{1}{2}}(q;\frac{-p}{t^{\mu
}\left( 1-t\right) ^{\sigma }})dt,
\end{eqnarray*}%
where $\left( 1-zt\right) ^{-a}=\sum_{n=0}^{\infty }(a)_{n}\frac{\left(
zt\right) ^{n}}{n!},\forall $ $\left\vert zt\right\vert <1.$ Similarly, from
the definitions of the functions (\ref{E1}) and (\ref{F2}), we can derive
the integral representation (\ref{F2I}) with $\exp (zt)=\sum_{n=0}^{\infty }%
\frac{\left( zt\right) ^{n}}{n!}.$
\end{proof}

Also, it can be easily seen that the new extensions (\ref{F1}) and (\ref{F2I}%
) reduce to the following special functions as\newline

\begin{case}
\textit{extended Gauss hypergeometric function and extended confluent
hypergeometric function in Lee et al. (\cite{LRPK}, pp. 189, Equations.
(1.11) and (1.12)]) respectively when $\upsilon =0$, $q=\frac{1}{2}$ and $%
\mu =\sigma =1$.}
\end{case}

\begin{case}
new \textit{generalized beta function in \"{O}zergin et al. (\cite{OMA}; pp.
4607, Equations. (11)) when }$q=2\alpha -\beta +\frac{1}{2},v=\beta -\alpha -%
\frac{1}{2}$ and\textit{\ $\mu =\sigma =1$.}
\end{case}

\begin{remark}
An interesting generalization of extension of gamma function and generalized
gamma function given together in the paper (\textit{\cite{OMA})} can be
considered as%
\begin{eqnarray*}
\Gamma _{p}^{v,q}\left( x\right) &:&=\int_{0}^{1}t^{x-1}I_{v+\frac{1}{2}%
}(q;-(t+\frac{p}{t})dt, \\
(\func{Re}(p) &>&0,\func{Re}(x)>0,\func{Re}(v+q)>-\frac{1}{2},\func{Re}%
(2v+q)>-\frac{3}{2}).
\end{eqnarray*}
\end{remark}

\subsection{The Mellin and Laplace Transforms}

In this section, we derive the Mellin and Laplace transforms of extended
modified Bessel and extended beta-hypergeometric functions. The necessary
conditions for their existences can be followed through existences of the
special functions appearing in their respective formulae.

\begin{theorem}
The Mellin transform of 
\begin{eqnarray*}
M\left[ I_{v}(q;x);s\right] &:&=\int_{0}^{\infty }x^{s-1}I_{v}(q;x)dx \\
&:&=\frac{\left( -1\right) ^{v+s-1}2^{q-s-\frac{1}{2}}\Gamma \left(
v+s\right) \Gamma \left( q-s\right) }{\sqrt{\pi }\Gamma \left( q+v-s+\frac{1%
}{2}\right) }
\end{eqnarray*}%
whenever integral exists.

\begin{proof}
Assume that the Mellin transform of $I_{v}(q;x)$ exists. Then, 
\begin{eqnarray*}
M\left[ I_{v}(q;x);s\right] &:&=\int_{0}^{\infty }x^{s-1}I_{v}(q;x)dx \\
&:&=\frac{2^{v+q-\frac{1}{2}}}{\sqrt{\pi }\Gamma \left( v+\frac{1}{2}\right) 
}\int_{0}^{\infty }x^{s-1}\int_{0}^{1}x^{v}t^{\upsilon +q-1}\left(
1-t\right) ^{\upsilon -\frac{1}{2}}\exp (2xt)\ dtdx.
\end{eqnarray*}%
By using uniform convergency of the integration with substitutions $\sigma
=-2xt$ and $\lambda =t$, we have 
\begin{equation*}
\int_{0}^{\infty }x^{s-1}I_{v}(q;x)dx=\frac{2^{q-s-\frac{1}{2}}\left(
-1\right) ^{v+s-1}}{\sqrt{\pi }\Gamma \left( v+\frac{1}{2}\right) }%
\int_{0}^{1}\lambda ^{q-s-1}\left( 1-t\right) ^{\upsilon -\frac{1}{2}%
}d\lambda \int_{0}^{\infty }\sigma ^{v+s-1}\exp (-\sigma )\ d\sigma .
\end{equation*}%
Hence,%
\begin{eqnarray*}
\int_{0}^{\infty }x^{s-1}I_{v}(q;x)dx &=&\frac{2^{q-s-\frac{1}{2}}\left(
-1\right) ^{v+s-1}\Gamma \left( v+s\right) }{\sqrt{\pi }\Gamma \left( v+%
\frac{1}{2}\right) }B\left( q-s,v+\frac{1}{2}\right) \\
&=&\frac{2^{q-s-\frac{1}{2}}\left( -1\right) ^{v+s-1}\Gamma \left(
v+s\right) \Gamma \left( q-s\right) }{\sqrt{\pi }\Gamma \left( q+v-s+\frac{1%
}{2}\right) }.
\end{eqnarray*}
\end{proof}
\end{theorem}

In the paper (\textit{\cite{ACJ}},pp. 410, equation (3.3)), the Mellin
transform of the hypergeometric function $_{1}F_{1}\left( \alpha ,\beta
,-b\right) $ is used as%
\begin{equation}
\int_{0}^{\infty }b^{s-1}\text{ }_{1}F_{1}\left( \alpha ,\beta ,-b\right) db=%
\frac{\Gamma \left( \alpha -s\right) \Gamma \left( \beta \right) \Gamma
\left( s\right) }{\Gamma \left( \alpha \right) \Gamma \left( \beta -s\right) 
}.  \label{M1F1}
\end{equation}%
By using Mellin transform of $I_{v}(q;x),$ the Mellin transform of $_{1}F_{1}
$ can easily be showed by the following corollary.

\begin{corollary}
From the Mellin transform of $I_{v}(q;x),$ we can easily derive the Mellin
transform of $_{1}F_{1}\left( \alpha ,\beta ,-p\right) $ as 
\begin{equation*}
\int_{0}^{\infty }p^{s-1}\text{ }_{1}F_{1}\left( \alpha ,\beta ,-p\right) dp=%
\frac{\Gamma \left( \alpha -s\right) \Gamma \left( \beta \right) \Gamma
\left( s\right) }{\Gamma \left( \alpha \right) \Gamma \left( \beta -s\right) 
}.
\end{equation*}

\begin{proof}
Considering the Mellin transform of $I_{v}(q;x)$ with relation (\ref{I-F}),
we have%
\begin{equation*}
\int_{0}^{\infty }{\scriptsize p}^{s-1}\left[ \frac{\left( p\right)
^{\upsilon }\ 2^{\upsilon +q-\frac{1}{2}}\ \Gamma (\upsilon +q)}{\sqrt{\pi }%
\ \Gamma (2\upsilon +q+\frac{1}{2})}{_{1}F_{1}}\left( \upsilon +q,2\upsilon
+q+\frac{1}{2},2p\right) \right] {\scriptsize dp=}\frac{2^{q-s-\frac{1}{2}%
}\left( -1\right) ^{v+s-1}\Gamma \left( v+s\right) \Gamma \left( q-s\right) 
}{\sqrt{\pi }\Gamma \left( q+v-s+\frac{1}{2}\right) }.
\end{equation*}%
If we consider the substitutions $p\rightarrow -\frac{p}{2},\alpha
=v+q,\beta =2v+q+\frac{1}{2}$ for above integration, we have%
\begin{equation*}
\int_{0}^{\infty }\frac{p^{(s+v)-1}\Gamma \left( \alpha \right) 2^{q-s-\frac{%
1}{2}}\left( -1\right) ^{v+s-1}}{\sqrt{\pi }\Gamma \left( \beta \right) }%
\text{ }_{1}F_{1}\left( \alpha ,\beta ,-p\right) dp=\frac{2^{q-s-\frac{1}{2}%
}\left( -1\right) ^{v+s-1}\Gamma \left( v+s\right) \Gamma \left( q-s\right) 
}{\sqrt{\pi }\Gamma \left( q+v-s+\frac{1}{2}\right) }
\end{equation*}%
which gives Mellin transform (\ref{M1F1}).
\end{proof}
\end{corollary}

\begin{theorem}
The Mellin transform of 
\begin{eqnarray*}
M\left[ B_{v,q}^{\left( \mu ,\sigma \right) }(x,y;p);s\right] 
&:&=\int_{0}^{\infty }p^{s-1}B_{v,q}^{\left( \mu ,\sigma \right) }(x,y;p)dp
\\
&:&=\frac{2^{q-s}}{\pi }\frac{\left( -1\right) ^{v}\Gamma \left( v+s\right)
\Gamma \left( q-s\right) }{\Gamma \left( q+v-s+\frac{1}{2}\right) }B\left(
x+\mu s+1,y+\sigma s+1\right) 
\end{eqnarray*}%
whenever integral exists.

\begin{proof}
In the light of Mellin transform of $I_{v}(q;x),$ the Mellin transform of $%
B_{v,q}^{\left( \mu ,\sigma \right) }(x,y;p)$ can be represented as 
\begin{equation*}
M\left[ B_{v,q}^{\left( \mu ,\sigma \right) }(x,y;p);s\right] =\sqrt{\frac{2%
}{\pi }}\int_{0}^{1}t^{x}\left( 1-t\right) ^{y}\int_{0}^{\infty }p^{s-1}%
\text{ }I_{v+\frac{1}{2}}(q;-\frac{p}{t^{\mu }\left( 1-t\right) ^{\sigma }}%
)dpdt.
\end{equation*}%
Let $p=\Theta \left( t^{\mu }\left( 1-t\right) ^{\sigma }\right) $ and $%
\lambda =t$ $\left( dp=d\Theta \left( t^{\mu }\left( 1-t\right) ^{\sigma
}\right) \text{ and }d\lambda =dt\right) .$ Then, 
\begin{eqnarray*}
M\left[ B_{v,q}^{\left( \mu ,\sigma \right) }(x,y;p);s\right]  &=&\sqrt{%
\frac{2}{\pi }}\int_{0}^{1}\lambda ^{x+\mu s}\left( 1-\lambda \right)
^{y+\sigma s}d\lambda \left[ \frac{2^{q-s-\frac{1}{2}}\left( -1\right)
^{v}\Gamma \left( v+s\right) \Gamma \left( q-s\right) }{\sqrt{\pi }\Gamma
\left( q+v-s+\frac{1}{2}\right) }\right]  \\
&=&\sqrt{\frac{2}{\pi }}\left[ \frac{2^{q-s-\frac{1}{2}}\left( -1\right)
^{v}\Gamma \left( v+s\right) \Gamma \left( q-s\right) }{\sqrt{\pi }\Gamma
\left( q+v-s+\frac{1}{2}\right) }\right] \int_{0}^{1}\lambda ^{x+\mu
s}\left( 1-\lambda \right) ^{y+\sigma s}d\lambda  \\
&=&\frac{\sqrt{2}}{\pi }\frac{2^{q-s-\frac{1}{2}}\left( -1\right) ^{v}\Gamma
\left( v+s\right) \Gamma \left( q-s\right) }{\Gamma \left( q+v-s+\frac{1}{2}%
\right) }B\left( x+\mu s+1,y+\sigma s+1\right) .
\end{eqnarray*}
\end{proof}
\end{theorem}

\begin{remark}
Since the special function (\ref{E1}) is extension of some recently
introduced special functions, the Mellin transform of these covered
functions can be derived.
\end{remark}

\begin{theorem}
The Laplace transform, if exists, of extended modified Bessel function is%
\begin{eqnarray}
L\{I_{v}(q;x);s\} &:&=\int_{0}^{\infty }e^{-sx}I_{v}(q;x)dx  \label{LI} \\
&:&=\frac{2^{q+v-\frac{1}{2}}\Gamma \left( v+q\right) \Gamma \left(
v+1\right) }{\sqrt{\pi }s^{v+1}\Gamma \left( q+2v+\frac{1}{2}\right) }%
F\left( v+1,v+q;2v+q+\frac{1}{2};\frac{2}{s}\right) .  \notag
\end{eqnarray}%
where $F(a,b;c;z)$ is Gauss hypergeometric function (see (\textit{\cite{NVG}}%
,pp.11, equation (2)).

\begin{proof}
Consider the Laplace transform of $I_{v}(q;x)$ 
\begin{equation*}
L\{I_{v}(q;x);s\}=\int_{0}^{\infty }e^{-sx}I_{v}(q;x)dx=\int_{0}^{\infty
}e^{-sx}\frac{\left( \frac{x}{2}\right) ^{\upsilon }2^{2\upsilon +q-\frac{1}{%
2}}}{\sqrt{\pi }\ \Gamma \left( \upsilon +\frac{1}{2}\right) }%
\int_{0}^{1}t^{\upsilon +q-1}\left( 1-t\right) ^{\upsilon -\frac{1}{2}}\exp
(2xt)\ \emph{d}tdx.
\end{equation*}%
By using uniform convergency of the integration, we have%
\begin{equation*}
\int_{0}^{\infty }e^{-sx}I_{v}(q;x)dx=\frac{2^{\upsilon +q-\frac{1}{2}}}{%
\sqrt{\pi }\ \Gamma \left( \upsilon +\frac{1}{2}\right) }\int_{0}^{1}t^{%
\upsilon +q-1}\left( 1-t\right) ^{\upsilon -\frac{1}{2}}\emph{d}%
t\int_{0}^{\infty }\left( x\right) ^{\upsilon }e^{x(-s+2t)}\ dx.(s>2t)
\end{equation*}%
By using the substitutions $x\rightarrow \frac{\sigma }{s-2t}$ and $\lambda
=t,$ we have 
\begin{eqnarray*}
\int_{0}^{\infty }e^{-sx}I_{v}(q;x)dx &=&\frac{2^{\upsilon +q-\frac{1}{2}}}{%
\sqrt{\pi }\ \Gamma \left( \upsilon +\frac{1}{2}\right) }\int_{0}^{1}\lambda
^{\upsilon +q-1}\left( 1-\lambda \right) ^{\upsilon -\frac{1}{2}}\left(
s-2\lambda \right) ^{-v-1}\emph{d}\lambda \int_{0}^{\infty }\left( \sigma
\right) ^{\upsilon }e^{-\sigma }\ d\sigma . \\
&=&\frac{2^{\upsilon +q-\frac{1}{2}}\Gamma \left( \upsilon +1\right) }{\sqrt{%
\pi }\ s^{v+1}\Gamma \left( \upsilon +\frac{1}{2}\right) }%
\int_{0}^{1}\lambda ^{\upsilon +q-1}\left( 1-\lambda \right) ^{\upsilon -%
\frac{1}{2}}\left( 1-\frac{2}{s}\lambda \right) ^{-v-1}\emph{d}\lambda .
\end{eqnarray*}%
Since 
\begin{equation*}
F\left( a,b;c;z\right) =\frac{1}{B(b,c-b)}\int_{0}^{1}t^{b-1}\left(
1-t\right) ^{c-b-1}\left( 1-z\lambda \right) ^{-a}\emph{d}t\text{ with }%
\left\vert \arg \left( 1-z\right) \right\vert <\pi ,
\end{equation*}%
then%
\begin{eqnarray*}
\int_{0}^{\infty }e^{-sx}I_{v}(q;x)dx &=&\frac{2^{\upsilon +q-\frac{1}{2}%
}\Gamma \left( \upsilon +1\right) }{\sqrt{\pi }s^{v+1}\ \Gamma \left(
\upsilon +\frac{1}{2}\right) }F(v+1,v+q;2v+q+\frac{1}{2};\frac{2}{s})\cdot
B(v+q,v+\frac{1}{2}) \\
&=&\frac{2^{\upsilon +q-\frac{1}{2}}\Gamma \left( \upsilon +1\right) }{\sqrt{%
\pi }\ s^{v+1}\Gamma \left( \upsilon +\frac{1}{2}\right) }F(v+1,v+q;2v+q+%
\frac{1}{2};\frac{2}{s})\frac{\Gamma \left( \upsilon +q\right) \Gamma \left(
\upsilon +\frac{1}{2}\right) }{\Gamma \left( 2\upsilon +q+\frac{1}{2}\right) 
}
\end{eqnarray*}%
which gives the formula (\ref{LI}).
\end{proof}
\end{theorem}

As an particular case, the Laplace transform of $I_{v}(q;x)$ (\ref{LI}) for $%
v=0$ and $q=\frac{1}{2}$ gives 
\begin{equation*}
L\{I_{0}(\frac{1}{2};x);s\}=\frac{1}{\sqrt{s^{2}-2s}}.
\end{equation*}%
Consequently, Laplace transform of modified Bessel function of the first
kind for $v=0$ via Laplace transform of $I_{v}(q;x)$ can easily be obtained 
\begin{equation*}
L\{I_{0}\left( x\right) ;s\}=L\{e^{-x}I_{0}(\frac{1}{2};x);s\}=\frac{1}{%
\sqrt{\left( s+1\right) ^{2}-2\left( s+1\right) }}=\frac{1}{\sqrt{s^{2}-1}}.
\end{equation*}

\begin{corollary}
The Laplace transform of modified Bessel function is 
\begin{equation*}
L\{I_{v}\left( x\right) ;s\}=\frac{1}{2^{v}\left( s+1\right) ^{v+1}\Gamma
\left( q+2v+\frac{1}{2}\right) }F\left( v+1,v+\frac{1}{2};2v+1;\frac{2}{s+1}%
\right) .
\end{equation*}

\begin{proof}
Assume that Laplace transform of $I_{v}(q;x)$ exists and equals to $F\left(
s\right) $. Consequently,%
\begin{equation*}
L\{I_{v}\left( x\right) ;s\}=L\{e^{-x}I_{v}\left( \frac{1}{2},x\right)
;s\}=F\left( s+1\right) .
\end{equation*}%
By using formula (\ref{LI}) together with Legendre's duplication formula, we
derive the corresponding formula.
\end{proof}
\end{corollary}

\section{\protect\bigskip Generalization of Fractional Derivatives}

\subsection{Extended Fractional Derivative via Extended Modified Bessel
Function}

\bigskip In this section, we introduce an interesting extended fractional
derivative which can be generalization of a large set of fractional
derivatives. Let $z>0$ then the new extension of Riemann-Liouville
fractional derivative $_{\mu ,\sigma }D_{v,q;z}^{\alpha ,\eta ,p}{f(z)}$ is
defined as follows: 
\begin{align}
& _{\mu ,\sigma }D_{v,q;z}^{\alpha ,\eta ,p}\left( {f(z)}\right) :=\frac{%
\sqrt{\frac{2}{\pi }}}{\Gamma (\alpha )}\int_{0}^{z}f(t)(z-t)^{\alpha
-1}t^{\eta }{}I_{v+\frac{1}{2}}(q;\frac{-pz^{^{\mu +\sigma }}}{t^{\mu
}\left( z-t\right) ^{\sigma }})dt,  \label{FD} \\
& (\min \{\func{Re}(\alpha )>0,\func{Re}(p),\func{Re}(\eta )>0,\func{Re}(v+q+%
\frac{1}{2})>0,\func{Re}(2v+q+\frac{3}{2})>0),\text{ }\mu ,\sigma \geq 0) 
\notag
\end{align}%
and $n-1<Re(\alpha )<n$ $(n=1,2,3,...)$.\newline

Now, we start with the extended fractional derivative of elementary function 
$f(z)=z^{\lambda }$.

\begin{corollary}
Let $\func{Re}(\eta +\lambda +\mu q)>-1$. Then 
\begin{equation}
_{\mu ,\sigma }D_{v,q;z}^{\alpha ,\eta ,p}(z^{\lambda })=\frac{z^{\eta
+\lambda +\alpha }}{\Gamma (\alpha )}B_{v,q}^{\left( \mu ,\sigma \right)
}(\eta +\lambda ,\alpha -1;p)  \notag
\end{equation}%
whenever the function $B_{v,q}^{\left( \mu ,\sigma \right) }$ exists.
\end{corollary}

\begin{proof}
Consider the fractional derivative (\ref{FD}), we get 
\begin{equation*}
_{\mu ,\sigma }D_{v,q;z}^{\alpha ,\eta ,p}(z^{\lambda })=\frac{\sqrt{\frac{2%
}{\pi }}}{\Gamma (\alpha )}\int_{0}^{z}(z-t)^{\alpha -1}t^{\eta +\lambda
}{}I_{v+\frac{1}{2}}(q;\frac{-pz^{^{\mu +\sigma }}}{t^{\mu }\left(
z-t\right) ^{\sigma }})dt.
\end{equation*}%
Taking $t=zu$, after a little simplification, gives 
\begin{align*}
_{\mu ,\sigma }D_{v,q;z}^{\alpha ,\eta ,p}(z^{\lambda })& =\frac{\sqrt{\frac{%
2}{\pi }}z^{\eta +\lambda +\alpha }}{\Gamma (\alpha )}\int_{0}^{1}(1-u)^{%
\alpha -1}u^{\eta +\lambda }{}I_{v+\frac{1}{2}}(q;\frac{-p}{u^{\mu }\left(
1-u\right) ^{\sigma }})du \\
& =\frac{z^{\eta +\lambda +\alpha }}{\Gamma (\alpha )}B_{v,q}^{\left( \mu
,\sigma \right) }(\eta +\lambda ,\alpha -1;p).
\end{align*}
\end{proof}

\begin{corollary}
\bigskip Let $\xi \neq 0$ and $\xi \in \mathbb{C}$. Then%
\begin{equation}
_{\mu ,\sigma }D_{v,q;z}^{\alpha ,\eta ,p}((z-\xi )^{r}):=\frac{(-\xi
)^{r}B(\eta ,\alpha -1)z^{\eta +\alpha }}{\Gamma (\alpha )}F_{v,q;p}^{(\mu
,\sigma )}(-r,\eta ;\eta +\alpha -1;\frac{z}{\xi })  \label{FP1}
\end{equation}%
whenever the function $F_{v,q;p}^{(\mu ,\sigma )}$ exists.
\end{corollary}

\begin{proof}
Consider the fractional derivative (\ref{FD}), we get 
\begin{align*}
& D_{z}^{\mu ,\eta ,p}((z-\xi )^{r}) \\
& =\frac{\sqrt{\frac{2}{\pi }}}{\Gamma (\alpha )}\int_{0}^{z}(z-t)^{\alpha
-1}(t-\xi )^{r}t^{\eta }{}I_{v+\frac{1}{2}}(q;\frac{-pz^{^{\mu +\sigma }}}{%
t^{\mu }\left( z-t\right) ^{\sigma }})dt \\
& =\frac{\sqrt{\frac{2}{\pi }}(-\xi )^{r}z^{\eta +\alpha }}{\Gamma (\alpha )}%
\int_{0}^{1}(1-u)^{\alpha -1}(1-\frac{z}{\xi }u)^{r}u^{\eta }{}I_{v+\frac{1}{%
2}}(q;\frac{-p}{u^{\mu }\left( 1-u\right) ^{\sigma }})du\text{ }(t=uz) \\
& =\frac{(-\xi )^{r}B(\eta ,\alpha -1)z^{\eta +\alpha }}{\Gamma (\alpha )}%
F_{v,q;p}^{(\mu ,\sigma )}(-r,\eta ;\eta +\alpha -1;\frac{z}{\xi }).
\end{align*}
\end{proof}

\bigskip The special case of new extension (\ref{FD}) with $p\rightarrow
2p,\mu =\sigma =1;v=0,q=\frac{1}{2};\alpha =-\mu -\frac{1}{2},\eta =\frac{-1%
}{2}$ reduces the generalized Riemann-Liouville fractional derivative which
is defined by \"{O}zarslan et al (\cite{MO}) as 
\begin{align}
& D_{z}^{\mu ,\eta ,p}{f(z)}:=\frac{1}{\Gamma (-\mu )}%
\int_{0}^{z}f(t)(z-t)^{-\mu -1}exp\big(\frac{-pz^{2}}{t(z-t)})dt, \\
& (Re(\mu )<0,Re(p)>0).  \notag
\end{align}

\bigskip Also, the particular case $\mu =\sigma =0;v=0,q=\frac{1}{2};\alpha
=-\mu ,\eta =0$ for extended fractional derivative (\ref{FD}) reduces the
Riemann-Liouville fractional derivative%
\begin{align}
& D_{z}^{\mu }{f(z)}:=\frac{1}{\Gamma (-\mu )}\int_{0}^{z}f(t)(z-t)^{-\mu
-1}dt, \\
& (Re(\mu )<0).  \notag
\end{align}

It is also important to note that the extended fractional derivative (\ref%
{FD}) reduces to extended fractional derivative 
\begin{align}
& I_{z}^{\mu ,b}\left\{ {f(z)}\right\} :=\frac{1}{\Gamma (\mu )}%
\int_{0}^{z}f(t)(z-t)^{\mu -1}{}_{1}F_{1}\left( \gamma ,\beta ,-\frac{%
bz^{\rho +\lambda }}{t^{\rho }(z-t)^{\lambda }}\right) dt, \\
& (\rho >0,\lambda >0,\min \{\func{Re}(\gamma ),\func{Re}(\beta ),\func{Re}%
(\mu ),\func{Re}(b)\}>0)  \notag
\end{align}%
defined in (\textit{\cite{LMA}},pp.647) when $p\rightarrow \frac{b}{2},\mu
=\rho ,\sigma =\lambda ;v=0,q=2\gamma -\beta +\frac{1}{2},v=\beta -\gamma -%
\frac{1}{2};\alpha =\mu +\beta \lambda -\gamma \lambda -\frac{\lambda }{2}%
,\eta =\beta \rho -\gamma \rho -\frac{\rho }{2}.$

Finally, Katugampola in the paper (\textit{\cite{K}) }introduced a new
fractional integral operator given by,%
\begin{equation}
\left( ^{\rho }I_{a+}^{\alpha }f\right) \left( x\right) =\frac{\rho
^{1-\alpha }}{\Gamma \left( \alpha \right) }\int_{a}^{x}\frac{\tau ^{\rho
-1}f\left( \tau \right) }{\left( x^{\rho }-\tau ^{\rho }\right) ^{1-\alpha }}%
d\tau ,  \label{KFD}
\end{equation}%
which is generalization of the Riemann-Liouville and the Hadamard fractional
integrals. The extended fractional derivative (\ref{FD}) reduces to the
fractional derivative (\ref{KFD}) when $z\rightarrow x^{\rho }-a^{\rho
},f\left( z\right) \rightarrow f\left( \left( z+a^{\rho }\right) ^{\frac{1}{%
\rho }}\right) ;\alpha \rightarrow \alpha ,\eta =0,v=0,q=\frac{1}{2}$ and $%
\mu =\sigma =0.$

In the light of these reductions, we can easily understand that the extended
fractional derivative (\ref{FD}) is generalization of many defined
fractional derivatives.

\subsection{\protect\bigskip Fractional derivative of Rational Functions}

In this section, we will derive the extended fractional derivative of
arbitrary rational functions. Consequently, the general representation of
fractional derivatives of many defined fractional derivatives of arbitrary
rational functions can firstly be derived.

Assume that $P(z)$ and $Q(z)$ are polynomials such that $\deg (P)<\deg (Q)$.
In this case, the real partial fraction decomposition of the rational
function $\frac{P(z)}{Q(z)}$ can be represented as 
\begin{equation}
\frac{P(z)}{Q(z)}=\displaystyle\sum_{i=1}^{p}\displaystyle%
\sum_{r=1}^{k_{_{i}}}\frac{a_{ir}}{(z-z_{_{i}})^{r}}+\displaystyle%
\sum_{j=1}^{q}\displaystyle\sum_{s=1}^{l_{_{j}}}\frac{\beta _{js}z+\gamma
_{js}}{\Big(z^{2}-2Re(z_{_{j}})z+|z_{_{j}}|^{2}\big)^{s}}  \label{B}
\end{equation}%
where $a_{ir},\beta _{js},\gamma _{js}\in \mathbb{R}$. In the representation
(\ref{B}), the inverse of quadratic functions can not be worked well in many
calculations. Because of this quadratic functions in the denominators, for
example,we can not derive the fractional derivatives of rational functions.
In this paper, \ we will use complex partial fraction decomposition method
together with formula (\ref{FP1}) to derive extended fractional derivatives
of rational functions. In the paper (\cite{OP}), the complex partial
fraction decomposition of arbitrary rational fraction was derived by the
following theorem:

\begin{theorem}
Let $x_{1},\ldots ,x_{p}$ be pairwise different real numbers and $%
z_{1},\ldots ,z_{q}\in \mathbb{C}\setminus \mathbb{R}$ be also pairwise
different. If $P(x)$ is a polynomial with real coefficients whose degree
satisfies the inequality $\deg \big(P(x)\big)<p+2(l_{1}+\cdots +l_{q})$,
then there exists $a_{ir},\beta _{js},\gamma _{js}\in \mathbb{R}$ and $%
b_{js}\in \mathbb{C}$ such that 
\begin{equation*}
\begin{array}{lll}
\frac{P(x)}{Q(x)} & = & \displaystyle\sum_{i=1}^{p}\displaystyle%
\sum_{r=1}^{k_{_{i}}}\frac{a_{ir}}{(x-x_{_{i}})^{r}}+\displaystyle%
\sum_{j=1}^{q}\displaystyle\sum_{s=1}^{l_{_{j}}}\frac{\beta _{js}x+\gamma
_{js}}{\big(x^{2}-2Re(z_{_{j}})x+|z_{_{j}}|^{2}\big)^{s}} \\ 
& = & \displaystyle\sum_{i=1}^{p}\displaystyle\sum_{r=1}^{k_{_{i}}}\frac{%
a_{ir}}{(x-x_{_{i}})^{r}}+\displaystyle\sum_{j=1}^{q}\displaystyle%
\sum_{s=1}^{l_{_{j}}}\Big(\frac{b_{js}}{\big(x-z_{_{j}}\big)^{s}}+\frac{\bar{%
b}_{js}}{\big(x-\bar{z}_{_{j}}\big)^{s}}\Big),%
\end{array}%
\end{equation*}%
where
\end{theorem}

\begin{equation*}
Q(x)=(x-x_{_{1}})^{k_{_{1}}}\ldots (x-x_{_{p}})^{k_{_{p}}}\big(%
x^{2}-2Re(z_{_{1}})x+|z_{_{1}}|^{2}\big)^{l_{_{1}}}\ldots \big(%
x^{2}-2Re(z_{_{q}})x+|z_{_{q}}|^{2}\big)^{l_{_{q}}}.
\end{equation*}%
The relations between the coefficients of the real partial fraction
decomposition and the coefficients of the complex partial fraction
decomposition are: 
\begin{equation}
\left. 
\begin{array}{lllllllll}
b_{j1} & = & \displaystyle\sum_{s=2}^{l}\beta _{js}\omega _{j}^{2}|\omega
|^{2(s-2)}C_{2s-3}^{s-2}+\displaystyle\sum_{s=1}^{l}(\beta _{js}\omega
_{j}z+\omega _{j}\gamma _{js})|\omega _{j}|^{2(s-1)}C_{2(s-1)}^{s-1}, &  & 
&  &  &  &  \\ 
\nonumber b_{j2} & = & \displaystyle\sum_{s=3}^{l}\beta _{js}\omega
_{j}^{3}|\omega _{j}|^{2(s-3)}C_{2s-4}^{s-3}+\displaystyle%
\sum_{s=2}^{l}(\beta _{js}\omega ^{2}z+\omega _{j}^{2}\gamma _{js})|\omega
_{j}|^{2(s-2)}C_{2s-3}^{s-2}, &  &  &  &  &  &  \\ 
& \vdots &  &  &  &  &  &  &  \\ 
b_{jl_{j}-1} & = & \beta _{jl_{j}}\omega ^{l}+\beta _{j\,l_{j}-1}\omega
_{j}^{l_{j}-1}z+\omega _{j}^{l_{j}-1}\gamma _{j\,l_{j}-1}+(\beta
_{jl_{j}}\omega _{j}^{l_{j}-1}z+\omega _{j}^{l_{j}-1}\gamma
_{jl_{j}})|\omega _{j}|^{2}C_{l_{j}}^{1} &  &  &  &  &  &  \\ 
b_{jl_{j}} & = & \beta _{jl_{j}}\omega _{j}^{l}z+\omega _{j}^{l}\gamma
_{jl_{j}}, &  &  &  &  &  & 
\end{array}%
\right.  \label{complex PFD}
\end{equation}%
where $\omega _{j}=\frac{1}{2iIm(z_{j})}$.

\begin{theorem}
Let $\func{Re}(\eta )>0$ and $\func{Re}(\alpha )>0.$ The extended fractional
derivative of arbitrary rational function satisfying previous theorem is%
\begin{align*}
& _{\mu ,\sigma }D_{v,q;z}^{\alpha ,\eta ,p}\left( \frac{P(z)}{Q(z)}\right) =
\\
& =\sum_{i=1}^{p}\displaystyle\sum_{r=1}^{k_{_{i}}}\frac{a_{ir}B(\eta
,\alpha -1)z^{\eta +\alpha }}{(-x_{_{i}})^{r}\Gamma (\alpha )}%
F_{v,q;p}^{(\mu ,\sigma )}(r,\eta ;\eta +\alpha -1;\frac{z}{x_{_{i}}}) \\
& +\sum_{j=1}^{q}\displaystyle\sum_{s=1}^{l_{_{j}}}\Big(b_{js}\frac{B(\eta
,\alpha -1)z^{\eta +\alpha }}{(-z_{j})^{s}\Gamma (\alpha )}F_{v,q;p}^{(\mu
,\sigma )}(s,\eta ;\eta +\alpha -1;\frac{z}{z_{j}}) \\
& +\bar{b}_{js}\frac{B(\eta ,\alpha -1)z^{\eta +\alpha }}{(-\bar{z}%
_{_{j}})^{s}\Gamma (\alpha )}F_{v,q;p}^{(\mu ,\sigma )}(s,\eta ;\eta +\alpha
-1;\frac{z}{\bar{z}_{_{j}}})\Big)
\end{align*}%
whenever the extended hypergeometric functions $F_{v,q;p}^{(\mu ,\sigma )}$
exist.
\end{theorem}

\begin{proof}
Considering the formula (\ref{FP1}) and complex partial fraction
decomposition, the proof of the theorem can easily be done.
\end{proof}

\bigskip A numerical example of extended fractional derivative of the
rational function given in (\cite{OP}) will be derived by the following
example.

\begin{example}
\bigskip Consider the rational function given in (\cite{OP})%
\begin{equation}
f\left( z\right) =\frac{2z+1}{\left( z^{2}+6z+10\right) ^{3}}.  \label{R1}
\end{equation}%
The complex partial fraction decomposition of function (\ref{F1}) can be
given as%
\begin{equation}
f(z)=\sum_{s=1}^{3}\left( \frac{z_{_{s}}}{\left( x-(-3+i)\right) ^{s}}+\frac{%
\bar{z}_{s}}{\left( x-(-3-i)\right) ^{s}}\right)  \label{D1}
\end{equation}%
where $z_{1}=\frac{15i}{16},z_{2}=\frac{15}{16}-\frac{i}{8}$ and $z_{3}=%
\frac{-1}{4}-\frac{5i}{8}$. By using the decomposition (\ref{D1}) , the
extended fractional derivative of rational function (\ref{R1}) can be given
as%
\begin{align*}
& _{\mu ,\sigma }D_{v,q;z}^{\alpha ,\eta ,p}(f(z))=\sum_{s=1}^{3}(z_{s}\frac{%
B(\eta ,\alpha -1)z^{\eta +\alpha }}{(3-i)))^{s}\Gamma (\alpha )}%
F_{v,q;p}^{(\mu ,\sigma )}(s,\eta ;\eta +\alpha -1;\frac{z}{i-3})+ \\
& +\bar{z}_{s}\frac{B(\eta ,\alpha -1)z^{\eta +\alpha }}{(3+i)))^{s}\Gamma
(\alpha )}F_{v,q;p}^{(\mu ,\sigma )}(s,\eta ;\eta +\alpha -1;\frac{z}{-i-3}))
\end{align*}
\end{example}

\subsection{Multiplicative Extended Fractional Derivative (MEFD)}

In this section, we analogously introduce the extended multiplicative
fractional derivative in the multiplicative sense. Recall that if $f$ is
positive and Riemann integrable on $[a,b]$, then it is $^{\ast }$integrable
(multiplicative integrable) on $[a,b]$ and 
\begin{equation}
\int_{a}^{b}f(x)^{dx}=e^{\int_{a}^{b}(\ln \circ f)(x)\,dx}.  \notag
\end{equation}%
Recently, Abdeljawad and Grossman in (\cite{AG}) introduced Caputo, Riemann
and Letnikov multiplicative fractional derivatives and gave their
properties. The multiplicative Riemann-Liouville fractional derivative for $%
f(z)$ of order $\alpha $, $Re(\alpha )>0$, $n-1<Re(\alpha )<n$, starting
from 0 is represented as 
\begin{equation}
{}_{\ast }D^{\alpha }(f(z))=e^{D^{\alpha }(ln(f(z)))}=e^{\frac{1}{\Gamma
(n-\alpha )}\frac{d^{n}}{dx^{n}}\int_{0}^{z}(\ln f(t))(z-t)^{n-\alpha -1}dt},
\notag
\end{equation}%
also it can be represented as 
\begin{equation}
{}_{\ast }D^{\alpha }(f(z))=e^{D^{\alpha }(ln(f(z)))}=e^{\frac{1}{\Gamma
(-\alpha )}\int_{0}^{z}(\ln f(t))(z-t)^{-\alpha -1}dt},Re(\alpha )<0.
\label{MFD}
\end{equation}%
For more details we refer the recent paper (\cite{AG}).

In the lights of multiplicative fractional derivative (\ref{MFD}) , we
introduce multiplicative extended fractional derivative (MEFD) as follows:%
\begin{equation}
{}_{\mu ,\sigma \ast }D_{v,q;z}^{\alpha ,\eta ,p}\left( {f(z)}\right) =e^{%
\frac{\sqrt{\frac{2}{\pi }}}{\Gamma (\alpha )}\int_{0}^{z}(\ln
f(t))(z-t)^{\alpha -1}t^{\eta }{}I_{v+\frac{1}{2}}(q;\frac{-pz^{^{\mu
+\sigma }}}{t^{\mu }\left( z-t\right) ^{\sigma }})dt},  \label{MGFD}
\end{equation}%
where $f$ is positive valued function. It is clear that the MEFD can be
easily reduced to multiplicative Riemann-Liouville fractional derivatives
given in \cite{AG} when $\mu =\sigma =0;v=0,q=\frac{1}{2};\alpha =-\mu ,\eta
=0.$

Consequently, the following theorem can be obtained for the relationship
between multiplicative and ordinary generalized fractional derivatives:

\begin{theorem}
Suppose that $f$ is a positive function on $I$. The multiplicative
generalized fractional derivative is exists if and only if ordinary
generalized fractional derivative exists and 
\begin{equation}
_{\mu ,\sigma }D_{v,q;z}^{\alpha ,\eta ,p}(f(z))=_{\text{ }\mu ,\sigma \ast
}D_{v,q;z}^{\alpha ,\eta ,p}(e^{f\left( z\right) }).  \label{R2}
\end{equation}
\end{theorem}

\begin{remark}
The existence of the extended fractional derivative of a positive function
is equivalent to existence of multiplicative fractional derivative of the
same function.
\end{remark}

\begin{remark}
If we set $\mu =\sigma =1;v=0,q=\frac{1}{2};\alpha =-\mu -\frac{1}{2},\eta =%
\frac{-1}{2}$ in (\ref{MGFD}), then it gives alternative extended
multiplicative fractional derivative%
\begin{equation}
{}_{\mu ,\sigma \ast }D_{v,q;z}^{\alpha ,\eta ,p}\left( {f(z)}\right) )=\exp
(\frac{\sqrt{\frac{2}{\pi }}}{\Gamma (\alpha )}\int_{0}^{z}\left( \ln
f(t)\right) (z-t)^{\alpha -1}exp\big(\frac{-pz^{2}}{t(z-t)}\big)dt)  \notag
\end{equation}%
which is analogous of generalized fractional derivative 
\begin{equation}
D_{z}^{\mu ,p}{f(z)}:=\frac{1}{\Gamma (\alpha )}\int_{0}^{z}f(t)(z-t)^{%
\alpha -1}exp\big(\frac{-pz^{2}}{t(z-t)}\big)dt  \notag
\end{equation}%
given in \cite{MO}.
\end{remark}

The following formulae can easily be derived for MEFD.

\begin{corollary}
Let $Re(\mu )<0$ and $Re(\lambda )>0.$Then%
\begin{equation}
{}_{\mu ,\sigma \ast }D_{v,q;z}^{\alpha ,\eta ,p}(e^{z^{\lambda }}):=\exp (%
\frac{z^{\eta +\lambda +\alpha }}{\Gamma (\alpha )}B_{v,q}^{\left( \mu
,\sigma \right) }(\eta +\lambda ,\alpha -1;p)).
\end{equation}
\end{corollary}

\begin{corollary}
Let $a>0$ and $\xi \neq 0.$Then,%
\begin{equation}
{}_{\mu ,\sigma \ast }D_{v,q;z}^{\alpha ,\eta ,p}(a^{(z-\xi )^{r}}):=e^{%
\frac{\ln a(-\xi )^{r}B(\eta ,\alpha -1)z^{\eta +\alpha }}{\Gamma (\alpha )}%
F_{v,q;p}^{(\mu ,\sigma )}(-r,\eta ;\eta +\alpha -1;\frac{z}{\xi }).}.
\end{equation}
\end{corollary}

\bigskip Now, we give some properties of MEFD given in (\ref{MGFD}). Let $f$
and $g$ be $^{\ast }$integrable on $[a,b].$ Then, the following
properties/rules of $_{\mu ,\sigma \ast }D_{v,q;z}^{\alpha ,\eta ,p}(f(z))$
can be ordered as: 
\begin{align*}
\mathrm{(a)}& \ _{\mu ,\sigma \ast }D_{v,q;z}^{\alpha ,\eta
,p}(f(z)^{k})=\left( _{\mu ,\sigma \ast }D_{v,q;z}^{\alpha ,\eta
,p}(f(z))\right) ^{k},k\in \mathbb{R-}\left\{ 0\right\} \\
\mathrm{(b)}& \ _{\mu ,\sigma \ast }D_{v,q;z}^{\alpha ,\eta ,p}(f(z)g\left(
z\right) )=\left( _{\mu ,\sigma \ast }D_{v,q;z}^{\alpha ,\eta
,p}(f(z)\right) \cdot \left( _{\mu ,\sigma \ast }D_{v,q;z}^{\alpha ,\eta
,p}(g\left( z\right) \right) , \\
\mathrm{(c)}& \ _{\mu ,\sigma \ast }D_{v,q;z}^{\alpha ,\eta ,p}\bigg(\frac{%
f(x)}{g(x)}\bigg)=\frac{_{\mu ,\sigma \ast }D_{v,q;z}^{\alpha ,\eta ,p}(f(z))%
}{_{\mu ,\sigma \ast }D_{v,q;z}^{\alpha ,\eta ,p}(g(z))},\text{ }g(z)\neq 0,
\end{align*}%
The proofs of properties can be easily obtained from the definition of MEFD.
For example, the proof of the rule (b) follows from 
\begin{eqnarray*}
_{\mu ,\sigma \ast }D_{v,q;z}^{\alpha ,\eta ,p}(f(z)g\left( z\right) ) &=&e^{%
\frac{\sqrt{\frac{2}{\pi }}}{\Gamma (\alpha )}\int_{0}^{z}\left( \ln
f(t)+\ln g(t)\right) (z-t)^{\alpha -1}t^{\eta }{}I_{v+\frac{1}{2}}(q;\frac{%
pz^{^{\mu +\sigma }}}{t^{\mu }\left( z-t\right) ^{\sigma }}dt} \\
&=&_{\mu ,\sigma \ast }D_{v,q;z}^{\alpha ,\eta ,p}(f(z))\cdot _{\mu ,\sigma
\ast }D_{v,q;z}^{\alpha ,\eta ,p}(g(z))
\end{eqnarray*}

Next, we give the following theorem for the analytic function $f(z)$.

\begin{theorem}
Let $f(z)$ be an analytic function on an open interval $I$ for $\left\vert
z\right\vert <1$. If $f(z)$ has Maclaurin's series as
\end{theorem}

\begin{equation}
f(z)=\sum_{k=0}^{\infty }a_{k}z^{k},  \notag
\end{equation}%
then from the uniformly convergence of the integral

\begin{equation*}
{}_{\mu ,\sigma \ast }D_{v,q;z}^{\alpha ,\eta ,p}(e^{f(z)})=e^{\frac{1}{%
\Gamma (\alpha )}\sum_{k=0}^{\infty }a_{k}z^{\eta +k+\alpha }B_{v,q}^{\left(
\mu ,\sigma \right) }(\eta +k,\alpha -1;p)}.
\end{equation*}

\bigskip

\begin{proof}
Suppose that \bigskip $f(z)$ is analytic function over the interval $I$ and
it has series representation as $f(z)=\sum_{k=0}^{\infty }a_{k}z^{k}.$ 
\newline
Then,%
\begin{equation*}
_{\mu ,\sigma \ast }D_{v,q;z}^{\alpha ,\eta ,p}(e^{f(z)})=\exp (\frac{\sqrt{%
\frac{2}{\pi }}}{\Gamma (\alpha )}\sum_{k=0}^{\infty }{\large a}%
_{k}\int_{0}^{z}{\large (z-t)}^{\alpha -1}{\large t}^{\eta +k}{\large {}I_{v+%
\frac{1}{2}}(q;\frac{-pz^{^{\mu +\sigma }}}{t^{\mu }\left( z-t\right)
^{\sigma }})dt}).
\end{equation*}%
Choosing $t=zu$ and interchanging the summation and integral which is
quarentee, then 
\begin{equation*}
_{\mu ,\sigma \ast }D_{v,q;z}^{\alpha ,\eta ,p}(e^{f(z)})=\exp (\frac{1}{%
\Gamma (\alpha )}\sum_{k=0}^{\infty }a_{k}z^{\eta +k+\alpha }B_{v,q}^{\left(
\mu ,\sigma \right) }(\eta +k,\alpha -1;p)).
\end{equation*}%
This completes the proof.\bigskip
\end{proof}

\section{Conclusion}

Recently, the investigation for extension of some special functions have
become important. Thus, many extensions of special functions have been
obtained by the authors in different studies. From this point of view, we
present extended modified Bessel function $I_{\upsilon }(q;x)$ which
generalizes the Bessel And modified Bessel functions, by using an additional
parameter in the integral representation. An extensions of the well-known
functions in the literature such as the hypergeometric, the confluent
hypergeometric and the extended beta functions are also given via extended
modified Bessel function. A necessary relation between extended modified
Bessel function $I_{\upsilon }(q;x)$ and the confluent hypergeometric
function ${_{1}F_{1}}(\alpha ,\beta ,x)$ is easily given. Moreover, Mellin
and Laplace transforms for some newly derived special functions are obtained
as a common coverage. We determine asymptotic formulae and also the
generating functions of the extended modified Bessel function. Hence, a lot
of relations with respect to this new function can be proved by using its
generating functions. In last section, we introduce new extensions of the
classical and multiplicative Riemann-Liouville fractional derivatives via
defined extended special function $I_{\upsilon }(q;x)$. The fractional
derivative of rational functions is explicitly found by using the new
definition of fractional derivative and complex partial fraction
decomposition. It can be easily seen that the results obtained in this paper
are new and effective mathematical tools and, also extensions of many
results in the literature.

\end{document}